\documentclass[12pt]{article}

\usepackage[latin1]{inputenc}
\usepackage[english]{babel}
\usepackage{amssymb,amsmath}
\usepackage{amsthm}
\usepackage{amsfonts}
\usepackage{graphicx}
\usepackage[margin=3.0cm]{geometry}
\usepackage{enumerate}
\usepackage{color}
\usepackage[noadjust]{cite}
\usepackage{tikz}
\usepackage{mathtools}
\usepackage{framed}
\usepackage{subfig}
\usepackage{hyperref}

\newtheorem{thm}{Theorem}

\newtheorem{cor}[thm]{Corollary}
\newtheorem{conj}[thm]{Conjecture}
\newtheorem{lem}[thm]{Lemma}
\newtheorem{prop}[thm]{Proposition}

\theoremstyle{definition}

\newtheorem*{conj*}{Conjecture}

{
  \begin{framed}\sffamily \textbf{#1}%
}%
{\rmfamily\end{framed}}

\newcommand{\Path}[1]{\ensuremath{{\sf P}_{#1}}}
\newcommand{\Q}[1]{\ensuremath{{\sf Q}_{#1}}}

\newcommand{\PathBar}[1]{\ensuremath{{\sf \overline{P}}_{#1}}}

\newcommand{\C}[1]{\ensuremath{{\sf C}_{#1}}}
\newcommand{\CBar}[1]{\ensuremath{{\sf \overline{C}}_{#1}}}

\newcommand{\K}[1]{\ensuremath{{\sf K}_{#1}}}
\newcommand{\B}[1]{\ensuremath{{\sf B}_{#1}}}
\newcommand{\TBell}[1]{\ensuremath{{\sf T}_{#1}}}
\newcommand{\ABell}[1]{\ensuremath{{\sf A}_{#1}}}
\newcommand{\E}[1]{\ensuremath{\overline{{\sf K}}_{#1}}}

\newcommand{\GPH}{\emph{GraPHedron}}
\newcommand{\PHOEG}{\emph{PHOEG}}
\newcommand{\ncol}{\ensuremath{\mathcal{B}}}
\newcommand{\avcol}{\ensuremath{\mathcal{A}}}
\newcommand{\totcol}{\ensuremath{\mathcal{T}}}
\newcommand{\ncolk}{\ensuremath{S}}

\newcommand{\chrom}{\ensuremath{\chi}}

\newcommand{\stirlings}[2]{\genfrac\{\}{0pt}{}{#1}{#2}}
\newcommand{\T}{\totcol}
\newcommand{\A}{\avcol}

\newcommand{\merge}[2]{\ensuremath{#1_{\mid #2}}}

\usetikzlibrary{shapes, graphs, graphs.standard, calc}
\tikzstyle{vertex}=[circle, draw, fill=black, minimum size=8pt, inner sep=0pt]

\makeatletter  

\title{Lower Bounds and properties for the average number of colors in the non-equivalent colorings of a graph}

\author{
		Alain Hertz\textsuperscript{1},
		Hadrien M\'elot\textsuperscript{2},
		S\'ebastien Bonte\textsuperscript{2},
		Gauvain Devillez\textsuperscript{2},
		\\[3mm]
		\footnotesize \textsuperscript{1} Department of Mathematics and Industrial
	Engineering\\
	\footnotesize Polytechnique Montr\'eal - Gerad, Montr\'eal, Canada\\
	\footnotesize Corresponding author. Email: alain.hertz@gerad.ca\\[3mm]
	 \footnotesize \textsuperscript{2} Computer Science Department - Algorithms Lab\\
	 \footnotesize University of Mons, Mons, Belgium
	}

\begin{document}

\maketitle
\vspace*{0.2cm}

\hrule
\vspace*{0.2cm}
\small
\noindent
\textbf{Abstract.} \\
We study the average number $\A(G)$ of colors in the non-equivalent colorings of a graph $G$. We show some general properties of this graph invariant and determine its value for some classes of graphs. We then conjecture several lower bounds on $\A(G)$ and prove that these conjectures are true for specific classes of graphs such as triangulated graphs and graphs with maximum degree at most 2.

\vspace*{0.2cm}

\vspace*{0.2cm}
\hrule

\normalsize

\section{Introduction} \label{sec_intro}

A coloring of a graph $G$ is an assignment of colors to its vertices such that adjacent vertices have different colors. The chromatic number $\chi(G)$ is the minimum number of colors in a coloring of $G$. The total number $\ncol(G)$ of non-equivalent colorings (i.e., with different partitions into color classes) of a graph $G$ is the number of partitions of the vertex set of $G$ whose blocks are stable
sets (i.e., sets of pairwise non-adjacent vertices). This invariant has been studied by several authors in the last few years \cite{absil,Duncan10,DP09,GT13,Hertz16,KN14} under the name of (graphical) Bell number. 

Recently, Hertz et {\emph al.} have defined a new graph invariant $\avcol(G)$ 
which is equal to the  average number of colors in the non-equivalent  colorings of a graph $G$.
It can be seen as a generalization of a concept linked to Bell numbers. More precisely, the Bell numbers $(\B{n})_{n\geq 0}$ count the number of different ways to partition a set that has exactly $n$ elements. The 2-Bell numbers $(\TBell{n})_{n\geq 0}$ count the total number of blocks in all partitions of a set of $n$ elements. Odlyzko and Richmond \cite{OR} have studied the average number $\ABell{n}$ of blocks in a partition of a set of $n$ elements, which can be defined as $\ABell{n}=\frac{\TBell{n}}{\B{n}}.$
The graph invariant $\avcol(G)$ that we study in this paper generalizes $\ABell{n}$. Indeed, when constraints (represented by edges in $G$) impose that certain pairs of elements (represented by vertices) cannot belong to the same block of a partition, $\avcol(G)$ is the average number of blocks in the partitions that respect all constraints. Hence, for a graph of order $n$, $\avcol(G)=\ABell{n}$ if $G$ is the empty graph of order $n$.

As shown in \cite{Hertz21}, $\A(G)$ can help discover nontrivial inequalities for the Bell numbers. For example, we will see that $\avcol(\Path{n})=\frac{\B{n}}{\B{n-1}}$ and $\avcol(\Path{n})<\avcol(\Path{n+1})$ for $n\geq 1$, where $\Path{n}$ is the path on $n$ vertices. This immediately implies $\B{n}^2<\B{n-1}\B{n+1}$, which means that the sequence $(\B{n})_{n\geq 0}$ is strictly log-convex. This result has also been proved recently by Alzer~\cite{Alzer19} using numerical arguments. 

The best possible upper bound for $\avcol(G)$ is clearly the order $n$ of $G$ since all colorings of $G$ use at most $n$ colors and $\avcol(\K{n})=n$ for the clique $\K{n}$ of order $n$. It seems however much more complex to define a lower bound for $\avcol(G)$, as a function of $n$, which is reached by at least one graph of order $n$. We think that the best possible lower bound is $\frac{\B{n+1} - \B{n}}{\B{n}}$ and is reached by the empty graph of order $n$. But it is just a conjecture we are trying to prove.

In the next section we fix some notations, while Section \ref{sec_prop} is devoted to basic properties of $\avcol(G)$.
In Section \ref{sec:values}, we give values of $\avcol(G)$ for some particular graphs $G$ that we will deal with later. We then state in Section \ref{sec_lb} three conjectures on a lower bound for $\avcol(G)$ and prove that they are true 
for graphs $G$ with maximum degree $\Delta(G)\leq 2$ and for triangulated graphs.

\section{Notation} \label{sec_nota}

For basic notions of graph theory that are not defined here, we refer to Diestel~\cite{Diestel00}. Let $G = (V,E)$ be a simple undirected graph. The \emph{order} $n = |V|$ of $G$ is its number of vertices and the \emph{size} $m=|E|$ of $G$ is its number of edges.  We write $G \simeq H$ if $G$ and $H$ are two isomorphic graphs, and $\overline{G}$ is the complement of $G$. We denote by $\K{n}$ (resp. $\C{n}$, $\Path{n}$ and $\E{n}$) the \emph{complete graph} (resp. the \emph{cycle}, the \emph{path} and the empty graph) of order $n$. We write $\K{a, b}$ for the complete bipartite graph where $a$ and $b$ are the cardinalities of the two sets of vertices of the bipartition. For a subset $S$ of vertices in a graph $G$, we write $G[S]$ for the subgraph of $G$ induced by $S$.

Let $N(v)$ be the set of neighbors of a vertex $v$ in $G$. A vertex $v$ is \emph{isolated} if $|N(v)| = 0$ and is \emph{dominant} if $|N(v)| = n-1$ (where $n$ is the order of $G$). We write $\Delta(G)$ for the \emph{maximum degree} of $G$. A vertex $v$ of a graph $G$ is \emph{simplicial} if the induced subgraph $G[N(v)]$ of $G$ is a clique. A graph is \emph{triangulated} if each of its induced subgraphs contains a simplicial vertex.

Let $u$ and $v$ be two vertices in a graph $G$ of order $n$, We use the following notations:
\begin{itemize}\setlength\itemsep{-3pt}
	\item $\merge{G}{uv}$ is the graph (of order $n-1$) obtained from $G$ by identifying (merging) the vertices $u$ and $v$ and, if $uv \in E(G)$, by removing the edge $uv$;
	\item if $uv \in E(G)$, $G - uv$ is the graph obtained by removing the edge $uv$ from $G$;
	\item if $uv \notin E(G)$, $G + uv$ is the graph obtained by adding the edge $uv$ in $G$;
	\item $G - v$ is the graph obtained from $G$ by removing $v$ and all its incident edges.
\end{itemize}

 Given two graphs $G_1$ and $G_2$ (with disjoint sets of vertices), we write $G_1\cup G_2$ for the \emph{disjoint union} of $G_1$ and $G_2$, while the \emph{join} $G_1 + G_2$ of $G_1$ and $G_2$ is the graph obtained from  $G_1\cup G_2$ by adding all possible edges between the vertices of $G_1$ and those of $G_2$. Also, $G\cup p\K{1}$ is the graph obtained from $G$ by adding $p$ isolated vertices.

A \emph{coloring} of a graph $G$ is an assignment of colors to the vertices of $G$ such that adjacent vertices have different colors. The \emph{chromatic number} $\chrom(G)$ of $G$ is the minimum number of colors in a coloring of $G$. Two colorings are \emph{equivalent} if they induce the same partition of the vertex set into color classes. Let $\ncolk(G,k)$ be the number of non-equivalent colorings of a graph $G$ that use \emph{exactly} $k$ colors. Then, the total number $\ncol(G)$ of non-equivalent colorings of a graph $G$ is defined by 
$$\ncol(G) = \sum_{k = 1}^n \ncolk(G, k)=\sum_{k = \chrom(G)}^n \ncolk(G, k),$$
and the total number $\totcol(G)$ of color classes in the non-equivalent colorings of a graph $G$ is defined by
$$\totcol(G) = \sum_{k = 1}^n k \ncolk(G, k)=\sum_{k = \chrom(G)}^n k \ncolk(G, k).$$
The average number $\A(G)$ of colors in the non-equivalent colorings of a graph $G$ can therefore be defined as
$$\avcol(G) =\frac{\totcol(G)}{\ncol(G)}.$$

Note that 
$\ncol(\E{n})=\B{n}, \totcol(\E{n})=\TBell{n}$, and $\avcol(\E{n})=\ABell{n}$. As another example, consider the complement $\overline{\ensuremath{{\sf P}}}_5$ of a path on 5 vertices. As shown in Figure \ref{fig:PB5}, there are three non-equivalent colorings of $\overline{\ensuremath{{\sf P}}}_5$ with 3 colors, four with 4 colors, and one with 5 colors, which gives $\ncol(\overline{\ensuremath{{\sf P}}}_5)=8$, $\totcol(\overline{\ensuremath{{\sf P}}}_5)=30$ and $\A(\overline{\ensuremath{{\sf P}}}_5)=\frac{30}{8}=3.75.$

\begin{figure}[!hbtp]
	\centering
	\includegraphics[scale = 0.83]{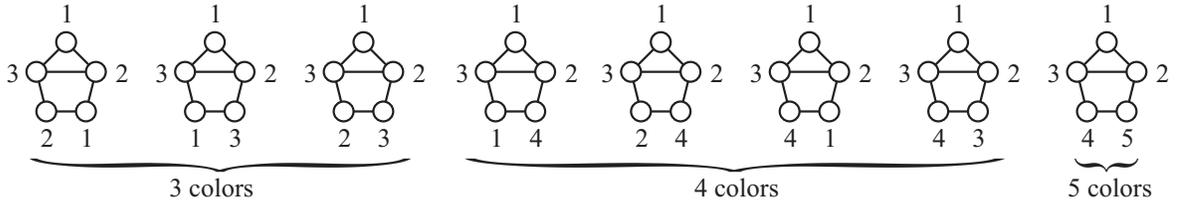}
	\caption{The non-equivalent colorings of $\overline{\ensuremath{{\sf P}}}_5$.}\label{fig:PB5}
\end{figure}

\section{Basic properties of $\A(G)$} \label{sec_prop}
 \allowdisplaybreaks
As for several other invariants in graph coloring, the \emph{deletion-contraction} rule (also often called the {\it Fundamental Reduction Theorem}~\cite{DKT05}) can be used to compute $\ncol(G)$ and $\totcol(G)$.  More precisely, let $u$ and $v$ be any pair of distinct vertices of $G$. As shown in~\cite{DP09,KN14}, we have 
\begin{align}
S(G, k) = S(G - uv, k) - S(G_{\mid uv}, k)& \quad \forall uv \in E(G),\label{recs_minus}\\
S(G, k) = S(G + uv, k) + S(G_{\mid uv}, k) & \quad \forall uv \notin E(G)\label{recs_plus}.
\end{align}
It follows that
\begin{align}
\left.
\begin{array}{ll}
\ncol(G) = \ncol(G - uv) - \ncol(\merge{G}{uv})\\
\T(G) = \T(G - uv) - \T(\merge{G}{uv})
\end{array}
\right\}& \quad \forall uv \in E(G),\label{rec_minus}\\
\left.
\begin{array}{ll}
\ncol(G) = \ncol(G + uv) + \ncol(\merge{G}{uv})\\
\T(G) = \T(G + uv) + \T(\merge{G}{uv})
\end{array}
\right\}& \quad \forall uv \notin E(G).\label{rec_plus}
\end{align}

\begin{thm} \label{thm:avColJoin}
	Given any two graphs $G_1$ and $G_2$, we have
	$$
	\avcol(G_1+G_2) = \avcol(G_1) + \avcol(G_2).
	$$
\end{thm}

\begin{proof}
	As observed in \cite{absil}, given any coloring of $G_1+G_2$, none of the vertices of $G_1$ can share a color with a vertex of $G_2$, which immediately gives $\ncol(G_1+G_2)=\ncol(G_1)\ncol(G_2)$. For $\totcol(G_1+G_2)$, assuming that $G_1$ and $G_2$ are of order $n_1$ and $n_2$, respectively, we get
	\begin{align*}
	\totcol(G_1+G_2)&=\sum_{k=1}^{n_1}\sum_{k'=1}^{n_2}(k+k')\ncolk(G_1,k)\ncolk(G_2,k')=\sum_{k=1}^{n_1}\ncolk(G_1,k)
	\sum_{k'=1}^{n_2}(k+k')\ncolk(G_2,k')\\
	&=\sum_{k=1}^{n_1}\ncolk(G_1,k)\left(k\sum_{k'=1}^{n_2}\ncolk(G_2,k')+\sum_{k'=1}^{n_2}k'\ncolk(G_2,k')\right)\\
	&=\sum_{k=1}^{n_1}k\ncolk(G_1,k)\sum_{k'=1}^{n_2}\ncolk(G_2,k')+\sum_{k=1}^{n_1}\ncolk(G_1,k)\sum_{k'=1}^{n_2}k'\ncolk(G_2,k')\\
	&=\totcol(G_1)\ncol(G_2)+\ncol(G_1)\totcol(G_2).
	\end{align*}
	Hence,
	\begin{align*}
	\avcol(G_1+G_2) &= \frac{\totcol(G_1+G_2)}{\ncol(G_1+G_2)}=\frac{\totcol(G_1) \ncol(G_2) +  \ncol(G_1)\totcol(G_2)}{\ncol(G_1) \ncol(G_2)}\\
	& = \frac{\totcol(G_1)}{\ncol(G_1)} + \frac{\totcol(G_2)}{\ncol(G_2)}  = \avcol(G_1) + \avcol(G_2).
	\end{align*}
\end{proof}
\noindent The following Corollary is also proved in \cite{Hertz21}.

\begin{cor}\label{pro_dominant}
If $v$ is a dominant vertex of a graph $G$, then,
$$
\avcol(G) = \avcol(G - v) + 1.
$$
\end{cor}
\begin{proof}
	If $v$ is a dominant vertex of a graph $G$, then $G\simeq(G-v)+\K{1}$, and since $\avcol(\K{1})=1$, Theorem \ref{thm:avColJoin} gives $\avcol(G)=\avcol(G-v)+1$.
\end{proof}

In the following, given a subset $W$ of vertices in a graph $G$, we denote by $\ncolk_{W,i}(G,k)$ the number of non-equivalent colorings of $G$ that use exactly $k$ colors, and where exactly $i$ of them appear on $W$. Hence, $\ncolk(G,k)=\sum_{i=0}^{|W|}\ncolk_{W,i}(G,k)$.

\begin{lem} \label{lem:AddNode}
Let $v$ be a vertex in a graph $G$ of order $n$ and let $N(v)$ be its set of neighbors in $G$. Then
\begin{itemize}\setlength\itemsep{-3pt}
\item $\displaystyle\ncol(G) = \ncol(G-v) + \sum_{k=1}^{n-1}\sum_{i=0}^{|N(v)|}(k-i)\ncolk_{N(v),i}(G-v,k)$, and
\item $\displaystyle\totcol(G) = \totcol(G-v) + \sum_{k=1}^{n-1}\sum_{i=0}^{|N(v)|}(k(k-i)+1)\ncolk_{N(v),i}(G-v,k)$.
\end{itemize} 
\end{lem}

\begin{proof}
Since $\ncolk_{N(v),i}(G,k) = \ncolk_{N(v),i}(G-v,k-1)+(k-i)\ncolk_{N(v),i}(G-v,k)$, we have
\begin{align*}
\ncol(G) &=\sum_{k=1}^{n}\ncolk(G,k)= \sum_{k=1}^{n}\sum_{i=0}^{|N(v)|}\ncolk_{N(v),i}(G,k)\\
&
= \sum_{k=1}^{n}\sum_{i=0}^{|N(v)|}\ncolk_{N(v),i}(G-v,k-1)+\sum_{k=1}^{n-1}\sum_{i=0}^{|N(v)|}(k-i)\ncolk_{N(v),i}(G-v,k)\\
&= \sum_{k=1}^{n-1}\sum_{i=0}^{|N(v)|}\ncolk_{N(v),i}(G-v,k)+\sum_{k=1}^{n-1}\sum_{i=0}^{|N(v)|}(k-i)\ncolk_{N(v),i}(G-v,k)\\
&= \ncol(G-v) + \sum_{k=1}^{n-1}\sum_{i=0}^{|N(v)|}(k-i)\ncolk_{N(v),i}(G-v,k)\\
\end{align*}

\vspace{-1.0cm}and
\begin{align*}
\totcol(G)&=\sum_{k=1}^{n}k\ncolk(G,k)= \sum_{k=1}^{n}\sum_{i=0}^{|N(v)|}k\ncolk_{N(v),i}(G,k)\\&
  = \sum_{k=1}^{n}\sum_{i=0}^{|N(v)|}k\ncolk_{N(v),i}(G-v,k-1)+\sum_{k=1}^{n-1}\sum_{i=0}^{|N(v)|}k(k-i)\ncolk_{N(v),i}(G-v,k)  \\
  &= \sum_{k=1}^{n-1}\sum_{i=0}^{|N(v)|}(k+1)\ncolk_{N(v),i}(G-v,k)+\sum_{k=1}^{n-1}\sum_{i=0}^{|N(v)|}k(k-i)\ncolk_{N(v),i}(G-v,k)\\
  &=
  \sum_{k=1}^{n-1}\sum_{i=0}^{|N(v)|}k\ncolk_{N(v),i}(G-v,k)+ \sum_{k=1}^{n-1}\sum_{i=0}^{|N(v)|}(k(k-i)+1)\ncolk_{N(v),i}(G-v,k)\\
  &= \totcol(G-v) + \sum_{k=1}^{n-1}\sum_{i=0}^{|N(v)|}(k(k-i)+1)\ncolk_{N(v),i}(G-v,k).\qedhere
\end{align*}

\end{proof}

\begin{thm}\label{thm:xhi}
Let $v$ be a vertex in a graph $G$. If $\chrom(G[N(v)])\geq|N(v)|-3$ then \linebreak[4]$\avcol(G)>\avcol(G-v)$.
\end{thm}

\begin{proof}

Let $n$ be the order of $G$. We know from Lemma~\ref{lem:AddNode} that
\begin{align*}
\avcol(G) - \avcol(G-v) = \frac{\totcol(G-v)+a}{\ncol(G-v)+b} - \frac{\totcol(G-v)}{\ncol(G-v)} 
					  = \frac{a\ncol(G-v) - b\totcol(G-v)}{\ncol(G)\ncol(G-v)}  
\end{align*}
where $\displaystyle a= \sum_{k=1}^{n-1}\sum_{i=0}^{|N(v)|}(k(k-i)+1)\ncolk_{N(v),i}(G-v,k)$ and
$\displaystyle b= \sum_{k=1}^{n-1}\sum_{i=0}^{|N(v)|}(k-i)\ncolk_{N(v),i}(G-v,k) $.\\

It suffices to show that $a\ncol(G{-}v){-}b\totcol(G{-}v)>0$. Let $\mathcal{P}$ be the set of pairs $(k,i)$ such that $\ncolk_{N(v),i}(G-v,k){>}0$. Since $\chrom(G[N(v)]) \geq |N(v)| - 3$, we have $k\geq i\geq |N(v)|-3$ for all $(k,i)\in\mathcal{P}$.
For two pairs $(k,i)$ and $(k',i')$ in $\mathcal{P}$, we write $(k,i){>}(k',i')$ if $k{>}k'$ or $k{=}k'$ and $i{>}i'$. Also, let $f(k,k',i,i')=\ncolk_{N(v),i}(G{-}v,k)\ncolk_{N(v),i'}(G{-}v,k')$. We have
\begin{align*}
& \quad a\ncol(G{-}v)-b\totcol(G{-}v) \\
=&\quad a\sum_{k=1}^{n-1}\sum_{i=0}^{|N(v)|}\ncolk_{N(v),i}(G-v,k)-b\sum_{k=1}^{n-1}\sum_{i=0}^{|N(v)|}k\ncolk_{N(v),i}(G-v,k)\\=&
\sum_{(k,i)\in \mathcal{P}}\ncolk_{N(v),i}(G-v,k)^2
\Bigl({(k(k-i)+1)-(k-i)k}\Bigr)\\ &\quad+\sum_{(k,i)>(k',i')}f(k,k',i,i')\Bigl({(k(k{-}i){+}1){+}(k'(k'{-}i'){+}1){-}(k{-}i)k'{-}(k'{-}i')k}\Bigr)\\
=&\quad\sum_{(k,i)\in \mathcal{P}}\ncolk_{N(v),i}(G{-}v,k)^2{+}
\sum_{(k,i){>}(k',i')}f(k,k',i,i')\Bigl({(k{-}k')^2{-}(k{-}k')(i{-}i'){+}2}\Bigr).
\end{align*}
Note that $\mathcal{P}\neq \emptyset$ since $\ncolk_{N(v),|N(v)|}(G-v,n-1)=1$. Hence,  $\sum_{(k,i)\in \mathcal{P}}\ncolk_{N(v),i}(G-v,k)^2{>}0$, and it is  sufficient to prove that $(k-k')^2-(k-k')(i-i')+2\geq 0$
for every two pairs $(k,i)$ and $(k',i')$ in $\mathcal{P}$ with $(k,i){>}(k',i')$. For two such pairs $(k,i)$ and $(k',i')$, we have $i-i'\leq |N(v)|-(|N(v)|-3)=3$. Hence,
\begin{itemize}\setlength\itemsep{-0.2em}
	\item if $k = k'$, then $(k-k')^2-(k-k')(i-i')+2 = 2 > 0$;
	\item if $k = k' + 1$, then $(k-k')^2-(k-k')(i-i')+2 = 3-(i-i') \geq 0$;
	\item if $k = k' + 2$, then $(k-k')^2-(k-k')(i-i')+2 = 6-2(i-i') \geq 0$;
	\item if $k \geq k' + 3$, then $(k-k')^2-(k-k')(i-i')+2 \geq 2$.\qedhere
\end{itemize}
\end{proof}

\begin{cor}\label{cor:5}
If $v$ is a vertex of degree at most 4 in a graph $G$, then $\avcol(G) > \avcol(G-v)$.
\end{cor}

\begin{proof} Since $|N(v)|\leq 4$, we have:
	\begin{itemize}\setlength\itemsep{-0.2em}
\item if $N(v) = \emptyset$, then $\chrom(G[N(v)]) = 0 > -3 = |N(v)| - 3$;
\item if $N(v) \neq \emptyset$, then $\chrom(G[N(v)]) \geq 1\geq  |N(v)| - 3$.
\end{itemize}
In both cases, we conclude from Theorem \ref{thm:xhi} that $\avcol(G) > \avcol(G-v)$.

\end{proof}

\begin{cor} \label{cor:addLinkedToClique}
Let $v$ be a simplicial vertex in a graph $G$. Then  $\avcol(G) > \avcol(G-v)$.
\end{cor}

\begin{proof}
Since $v$ is simplicial in $G$, we have $\chrom(G[N(v)])=|N(v)|>|N(v)|-3$. Hence, we conclude from Theorem \ref{thm:xhi} that  $\avcol(G) > \avcol(G-v)$.

\end{proof}

\begin{thm} \label{thm:removeSimplicialEdge}
Let $v$ be a simplicial vertex of degree at least one in a graph $G$ of order $n$, and let $w$ be one of its neighbors in $G$.
Then $\avcol(G)>\avcol(G-vw)$.
\end{thm}

\begin{proof}
Let $H=(G-v)\cup\K{1}$. In other words, $H$ is obtained from $G-v$ by adding an isolated vertex. It follows from Lemma~\ref{lem:AddNode} that 
	\begin{align*}\ncol(H) &= 
	\ncol(G-v)+
	\sum_{k=1}^{n-1}\sum_{i=0}^{0}(k-i)\ncolk_{\emptyset,i}(G-v,k)\\
	&=\ncol(G-v) + \sum_{k=1}^{n-1}k\ncolk(G-v,k)\end{align*}and
	\begin{align*}\totcol(H) &= 
	\totcol(G-v)+\sum_{k=1}^{n-1}\sum_{i=0}^{0}(k(k-i)+1)\ncolk_{\emptyset,i}(G-v,k)\\&=
	\totcol(G-v) + \sum_{k=1}^{n-1}(k^2+1)\ncolk(G-v,k).
	\end{align*}

Also, since $\ncolk_{N(v),i}(G{-}v,k){=}0$ for $i\neq |N(v)|$,  we have $\ncolk(G{-}v,k){=}\ncolk_{N(v),|N(v)|}(G{-}v,k)$ and it follows from Lemma~\ref{lem:AddNode} that 
\begin{align*}\ncol(G) &= \ncol(G-v) + \sum_{k=1}^{n-1}(k-|N(v)|)\ncolk(G-v,k)\\
&=\Big(\ncol(G-v) + \sum_{k=1}^{n-1}k\ncolk(G-v,k)\Big)-|N(v)|\sum_{k=1}^{n-1}\ncolk(G-v,k)\\
&=\ncol(H)-|N(v)|\ncol(G-v)\end{align*}and
\begin{align*}\totcol(G) &= \totcol(G-v) + \sum_{k=1}^{n-1}\Big(k(k-|N(v)|)+1\Big)\ncolk(G-v,k)\\
&= \Big(\totcol(G-v) + \sum_{k=1}^{n-1}(k^2+1)\ncolk(G-v,k)\Big)
-|N(v)|\sum_{k=1}^{n-1}
k\ncolk(G-v,k)\\
&=\totcol(H)-|N(v)|\totcol(G-v).\end{align*}
Similarly, since $v$ is simplicial (of degree $|N(v)|-1$) in $G-vw$, we have 
$$\ncol(G-vw)=\ncol(H)-(|N(v)|-1)\ncol(G-v)$$and
	$$\totcol(G-vw)=\totcol(H)-(|N(v)|-1)\totcol(G-v).$$
Hence,
\begin{align*}
\avcol(G)-\avcol(G-vw)&=
\frac{\totcol(G)}{\ncol(G)}-\frac{\totcol(G-vw)}{\ncol(G-vw)}\\
&=
\frac{\totcol(H){-}|N(v)|\totcol(G{-}v)}{\ncol(H){-}|N(v)|\ncol(G{-}v)}{-}\frac{\totcol(H){-}(|N(v)|{-}1)\totcol(G{-}v)}{\ncol(H){-}(|N(v)|{-}1)\ncol(G{-}v)}\\&=\frac{\totcol(H)\ncol(G{-}v){-}\ncol(H)\totcol(G{-}v)}{\ncol(G)\ncol(G{-}vw)}.
\end{align*}
Since $v$ is isolated in $H$, it is simplicial and we know from Corollary \ref{cor:addLinkedToClique} that
\begin{align*}\avcol(H)>\avcol(G-v)&\iff\frac{\totcol(H)}{\ncol(H)} > \frac{\totcol(G-v)}{\ncol(G-v)}\\& \iff \totcol(H)\ncol(G-v) - \ncol(H)\totcol(G-v) > 0
\end{align*}
which implies $\avcol(G)-\avcol(G-vw)>0$.

\end{proof}

\begin{lem} \label{lem:colUnion}

Let $G$ and $H$ be two graphs and suppose $H$ has order $n$. Then
\begin{itemize}
    \item $\ncol(G \cup H) = \displaystyle\sum_{k=1}^{n} \ncolk(H,k) \ncol(G \cup \K{k})$, and
    \item $\totcol(G \cup H) = \displaystyle\sum_{k=1}^{n} \ncolk(H,k) \totcol(G \cup \K{k})$.
\end{itemize}

\end{lem}
\begin{proof}
We first prove that $\ncol(G \cup H) = \sum_{k=1}^{n} \ncolk(H,k) \ncol(G \cup \K{k})$ for all graphs $H$ of order $n$. For $n=1$, we have $H = \K{1}$, and since $\ncolk(\K{1},1)=1$, we have $\ncol(G \cup \K{1}) = \sum_{k=1}^{1} \ncolk(\K{1},k) \ncol(G \cup \K{1})$. For larger values of $n$ we proceed by double induction on the order $n$ and the size $m$ of $H$. So assume $H$ has order $n$ and size $m$.
\begin{itemize}\setlength\itemsep{-0.1em}
	\item If $m=\frac{n(n-1)}{2}$, then $H=\K{n}$. Since $\ncolk(\K{n},i)=0$ for $i=1,\ldots,n-1$ and $\ncolk(\K{n},n)=1$, we have
	$\ncol(G \cup \K{n}) = \sum_{k=1}^{n} \ncolk(\K{n},k) \ncol(G \cup \K{n})$.
	\item If $m<\frac{n(n-1)}{2}$, then $H$ contains two non-adjacent vertices $u$ and $v$ and we know from Equations (\ref{rec_plus}) that $\ncol(G \cup H)=\ncol(G\cup(H + uv)) + \ncol(G\cup \merge{H}{uv})$. Since $G\cup(H + uv)$ has order $n$ and size $m+1$ and $G\cup(\merge{H}{uv})$ has order $n-1$, we know by induction that
	\begin{align*}
	\ncol(G \cup H)&=\sum_{k=1}^{n} \ncolk(H + uv,k) \ncol(G \cup \K{k}) + \sum_{k=1}^{n-1} \ncolk(\merge{H}{uv},k) \ncol(G \cup \K{k})\\
	&=\sum_{k=1}^{n} \bigl(\ncolk(H + uv,k) +\ncolk(\merge{H}{uv},k)\bigr) \ncol(G \cup \K{k})\\&=\sum_{k=1}^{n} \ncolk(H,k) \ncol(G \cup \K{k}).
	\end{align*}
\end{itemize}
\vspace{-0.3cm}The proof for $\totcol(G \cup H)$ is similar.

\end{proof}

\begin{thm} \label{thm:crossProductAvCol}
Let $H_1,H_2$ be any two graphs. If $\ncolk(H_1,k) \ncolk(H_2,k') \geq \ncolk(H_2,k)\ncolk(H_1,k')$ for all $k {>} k'$, the inequality being strict for at least one pair $(k,k')$, then\linebreak[4] ${\avcol(G\cup H_1) > \avcol(G\cup H_2)}$ for all graphs $G$.
\end{thm}

\begin{proof}
Let $f(k,k'){=}\ncolk(H_1,k) \ncolk(H_2,k') {-} \ncolk(H_2,k)\ncolk(H_1,k')$ and assume that $H_1$ and $H_2$ are of order $n_1$ and $n_2$, respectively. Note that $n_1\geq n_2$ else we would have $n_2>n_1$ and $f(n_2,n_1){=}\ncolk(H_1,n_2) \ncolk(H_2,n_1) {-} \ncolk(H_2,n_2)\ncolk(H_1,n_1)=-1<0$. We know from Lemma \ref{lem:colUnion} that
\begin{align*}
\avcol(G\cup H_1) {-} \avcol(G\cup H_2) &= \frac{\displaystyle\sum_{k=1}^{n_1}\ncolk(H_1,k)\totcol(G \cup \K{k})}{\displaystyle\sum_{k=1}^{n_1}\ncolk(H_1,k)\ncol(G \cup \K{k})} - \frac{\displaystyle\sum_{k=1}^{n_2}\ncolk(H_2,k)\totcol(G \cup \K{k})}{\displaystyle\sum_{k=1}^{n_2}\ncolk(H_2,k)\ncol(G \cup \K{k})}\\
&=\frac{\displaystyle\sum_{k=1}^{n_1}\sum_{k'=1}^{n_1}f(k,k')\totcol(G {\cup}\K{k})\ncol(G{\cup}\K{k'})}{\ncol(G\cup H_1)\ncol(G\cup H_2)}.
\end{align*}
Since $f(k,k)=0$ for all $k$ and $f(k,k')=-f(k',k)$ for all $k\neq k'$, we deduce 
\begin{align*}
\avcol(G\cup H_1) {-} \avcol(G\cup H_2) &=\frac{\displaystyle\sum_{k'=1}^{n_1-1}\sum_{k=k'{+}1}^{n_1}f(k,k')\Bigl(\totcol(G {\cup}\K{k})\ncol(G{\cup}\K{k'}){-}\totcol(G{\cup}\K{k'})\ncol(G{\cup}\K{k})\Bigr)}{\ncol(G\cup H_1)\ncol(G\cup H_2)}.
\end{align*}
Note that if $k>k'$, then $G\cup \K{k}$ is obtained from $G\cup \K{k'}$ by repeatedly adding a simplicial vertex. Hence, we know from Corollary \ref{cor:addLinkedToClique} that
\begin{align*}
\avcol(G\cup \K{k})>\avcol(G\cup \K{k'})
&\iff \frac{\totcol(G\cup \K{k})}{\ncol(G\cup \K{k})}>
\frac{\totcol(G\cup \K{k'})}{\ncol(G\cup \K{k'})}\\
&\iff 
\totcol(G \cup \K{k})\ncol(G \cup \K{k'})-\totcol(G \cup \K{k'})\ncol(G \cup \K{k})>0.
\end{align*}
Since $f(k,k')=\ncolk(H_1,k)\ncolk(H_2,k')-\ncolk(H_1,k')\ncolk(H_2,k)$ is positive for all $(k,k')$, and strictly positive for at least one such pair, we have $\avcol(G\cup H_1) - \avcol(G\cup H_2)> 0$.

\end{proof}

As a final property, we mention one which is proved in \cite{Hertz21} and which will be helpful in proving results in the following sections.

\begin{thm}[\cite{Hertz21}] \label{thm:pro_avcolInfSum}  Let $G, H$ and $F_1,\cdots,F_r$ be $r+2$ graphs, and let $\alpha_1,\cdots,\alpha_r$ be $r$ positive numbers such that
	\begin{itemize}\setlength\itemsep{0.1em}
		\item $\displaystyle\ncol(G)=\ncol(H)+\sum_{i=1}^r\alpha_i\ncol(F_i)$, 
		\item $\displaystyle\T(G)=\T(H)+\sum_{i=1}^r\alpha_i\T(F_i)$, and
		\item $\displaystyle\A(F_i)<\A(H)$ for all $i=1,\cdots,r$.
	\end{itemize}
	Then $\A(G)<\A(H)$.
\end{thm}

\section{Some values for $\A(G)$}\label{sec:values}

The value $\avcol(G)$ is known for some graphs $G$. 
We mention here some of them which are proven in \cite{Hertz21} and determine some others.

\begin{prop}\label{avcol-kn}\cite{Hertz21}
\begin{itemize}\setlength\itemsep{-0.5em}
	\item $\displaystyle\avcol(\E{n}) = \avcol(n\K{1})=\frac{\B{n+1} - \B{n}}{\B{n}}$ for all $n\geq 1$;
		\item $\displaystyle \avcol(T\cup p\K{1})=
	\frac{\displaystyle\sum_{i = 0}^{p} {p \choose i} \B{n+i}}
	{\displaystyle\sum_{i = 0}^{p} {p \choose i} \B{n+i-1}}$
	for all trees $T$ of order $n\geq 1$ and all $p\geq 0$;
	\item $\displaystyle\avcol(\C{n}\cup p\K{1}) = \frac{\displaystyle\sum_{j=1}^{n-1}(-1)^{j+1}\sum_{i = 0}^{p} {p \choose i} \B{n+i-j+1}}{\displaystyle\sum_{j=1}^{n-1}(-1)^{j+1}\sum_{i = 0}^{p} {p \choose i} \B{n+i-j}}$ for all $n\geq 3$ and $p\geq 0$.
\end{itemize}
\end{prop}

Since $\ncolk(\K{n},k)=1$ for $k=n$, and $\ncolk(\K{n},k)=0$ for $k<n$, we have $\avcol(\K{n})=n$. We prove here a stronger property which we use in the next section. Let $\stirlings{a}{b}$ be the Stirling number of the second kind, with parameters $a$ and $b$ (i.e., the number of partitions of a set of $a$ elements into $b$ blocks).
\begin{prop}\label{prop:KnpK1}
	$ $
	\newline
	$\quad\displaystyle\avcol(\K{n}\cup p\K{1})=\frac
	{\displaystyle\sum_{k=n}^{n+p}k\sum_{j=0}^n{k-j \choose n-j}{n \choose j} (n-j)!\stirlings{p}{k-j}}
	{\displaystyle\sum_{k=n}^{n+p}\sum_{j=0}^n{k-j \choose n-j}{n \choose j} (n-j)!\stirlings{p}{k-j}}\quad$ for all $n\geq 1$ and all $p\geq 0$.
	
\end{prop}
\begin{proof}
It is proved in \cite{Hertz16} that given two graphs $G_1$ and $G_2$, we have 
$$\ncolk(G_1\cup G_2,k)=\sum_{i=1}^k\sum_{j=0}^i{i \choose j}{k-j \choose i-j}(i-j)!\ncolk(G_1,i)\ncolk(G_2,k-j).$$
For $G_1\simeq \K{n}$ and $G_2\simeq p\K{1}$, we have $\ncolk(G_1,i)=1$ if $i=n$, and $\ncolk(G_1,i)=0$ otherwise. Also, $\ncolk(G_2,k-j)=\stirlings{p}{k-j}$. Hence, 
$$\ncolk(\K{n}\cup p\K{1},k)=\sum_{j=0}^n{k-j \choose n-j}{n \choose j} (n-j)!\stirlings{p}{k-j}.
$$
The result then follows from the fact that \begin{align*}\avcol(\K{n}\cup p\K{1})=\frac
{\displaystyle\sum_{k=n}^{n+p}k\ncolk(\K{n}\cup p\K{1},k)}
{\displaystyle\sum_{k=n}^{n+p}\ncolk(\K{n}\cup p\K{1},k)}.
\end{align*}
\end{proof}
We now determine $\avcol(G)$ for $G$ equal to the complement of a path and the complement of a cycle. In what follows, we write $F_n$ and $L_n$ for the $n$th Fibonacci number and the $n$th Lucas number, respectively.

\begin{prop} \label{prop:totColFibo}
	$\displaystyle\avcol(\PathBar{n}) =   \frac{(n+1)F_{n+2}+(2n-1)F_{n+1}}{5F_{n+1}}$ for all  $n\geq 1$.
\end{prop}
\begin{proof}
	The result is true for $n\leq 2$. Indeed, 
	\begin{itemize}
		\item For $n=1$, we have $\PathBar{1}=K_1$ which implies $\avcol(\PathBar{1})=1=\frac{2F_3+F_2}{5F_2}$;
		\item For $n=2$, we have $\PathBar{2}=\overline{K}_{2}$ which implies $\avcol(\PathBar{2})=\frac{\B{3} - \B{2}}{\B{2}}=\frac{3}{2}=\frac{3F_4+3F_3}{5F_3}$.
	\end{itemize}
	For larger values of $n$, we proceed by induction.
It is shown in \cite{DP09} that $\ncol(\PathBar{n}){=}F_{n+1}$.
Also, it follows from Equations \eqref{rec_plus} that 
$\totcol(\PathBar{n}) = \totcol(\PathBar{n-1} + \K{1}) + \totcol(\PathBar{n-2} + \K{1})$. Moreover, as shown in the proof of Theorem~\ref{thm:avColJoin}, we have
$$\totcol(G+\K{1}) = \totcol(G)\ncol(\K{1})+\ncol(G)\totcol(\K{1}) = \totcol(G) + \ncol(G).$$
Hence,
\begin{align*}
\avcol(\PathBar{n}) &= \frac{\totcol(\PathBar{n-1}) + \ncol(\PathBar{n-1}) + \totcol(\PathBar{n-2}) + \ncol(\PathBar{n-2})}{F_{n+1}} \\
&= \frac{nF_{n+1}+(2n-3)F_n}{5F_{n+1}} + \frac{F_n}{F_{n+1}} + \frac{(n-1)F_n+(2n-5)F_{n-1}}{5F_{n+1}} + \frac{F_{n-1}}{F_{n+1}} \\
&= \frac{nF_{n+1}+(3n+1)F_n+2nF_{n-1}}{5F_{n+1}} = \frac{3nF_{n+1}+(n+1)F_n}{5F_{n+1}} \\&= \frac{(n+1)F_{n+2} + (2n-1)F_{n+1}}{5F_{n+1}}.
\end{align*}
\end{proof}

\begin{prop} 
	$\displaystyle\avcol(\CBar{n}) =   \frac{nF_{n+1}}{L_n}$ for all  $n\geq 4$.
\end{prop}
\begin{proof}
	It follows from Equations \eqref{rec_plus} that 
	$$\totcol(\CBar{n}) = \totcol(\PathBar{n})+\totcol(\PathBar{n-2} + \K{1}).$$
	Moreover, it is shown in \cite{DP09} that $\ncol(\CBar{n})=L_{n}$. Since $\totcol(\PathBar{n-2} + \K{1})=\totcol(\PathBar{n-2})+\ncol(\PathBar{n-2})
	$, Proposition \ref{prop:totColFibo} implies
\begin{align*}
\avcol(\CBar{n}) &=  \frac{\totcol(\PathBar{n}) + \totcol(\PathBar{n-2}) + \ncol(\PathBar{n-2})}{L_n} \\
				  &= \frac{(n+1)F_{n+2} + (2n-1)F_{n+1}}{5L_n} + \frac{(n-1)F_{n} + (2n-5)F_{n-1}}{5L_n} + \frac{F_{n-1}}{L_n} \\
				  &= \frac{(n+1)F_{n+2}+(2n-1)F_{n+1}+(n-1)F_n+2nF_{n-1}}{5L_n} \\
				  &= \frac{3nF_{n+1}+2nF_{n}+2nF_{n-1}}{5L_n} = \frac{5nF_{n+1}}{5L_n} = \frac{nF_{n+1}}{L_n}.
\end{align*}

\end{proof}

    \section{Lower bounds on $\A(G)$} \label{sec_lb}
In this section, we give three conjectures for potential lower bounds on $\avcol(G)$. We then establish their validity for triangulated graphs and for graphs $G$ with maximum degree $\Delta(G)\le 2$.

\subsection{Conjectures}
 The lower bounds we are interested in depend on two parameters $n$ and $r$ with  $1{\le} r{\le} n$. They are equal to $\avcol(G)$ for some specific graphs $G$. More precisely, with the help of Propositions \ref{avcol-kn} and \ref{prop:KnpK1}, we define
\begin{itemize}
	\item $L_1(n)=\avcol(\E{n})=\frac{\displaystyle\B{n+1} - \B{n}}{\displaystyle\B{n}}$,
	\item $L_2(n,r)=\avcol(\K{r} \cup (n{-}r) \K{1})=\frac
	{\displaystyle\sum_{k=r}^nk\sum_{i=0}^r{k-i \choose r-i}{r \choose i} (r-i)!\stirlings{n-r}{k-i}}
	{\displaystyle\sum_{k=r}^n\sum_{i=0}^r{k-i \choose r-i}{r \choose i} (r-i)!\stirlings{n-r}{k-i}}$,
	\item $L_3(n,r)=\avcol(\K{1,r{-}1} \cup (n{-}r) \K{1})=\frac{\displaystyle\sum_{i = 0}^{n{-}r} {n{-}r \choose i} \B{r+i}}
	{\displaystyle\sum_{i = 0}^{n{-}r} {n{-}r \choose i} \B{r+i{-}1}}$.
\end{itemize} 

\captionsetup[subfigure]{labelformat=empty}\begin{figure}
	\centering
	\subfloat[][\E{7}]{
    	\includegraphics[width=.15\textwidth]{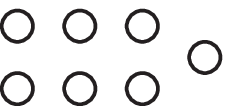}
	}
	\hspace{1cm}\subfloat[][$\K{7}\cup3\K{1}$]{
    	\includegraphics[width=.15\textwidth]{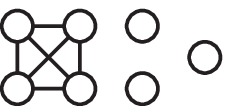}
	}
	\hspace{1cm}\subfloat[][$\K{1,3}\cup3\K{1}$]{
    	\includegraphics[width=.15\textwidth]{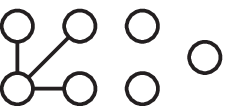}
	}
	\caption{The three graphs that define the lower bounds $L_1(7)$, $L_2(7,4)$ and $L_3(7,4)$.}\label{fig:3graphs}
\end{figure}


For illustration, we show in Figure \ref{fig:3graphs} the three graphs that give the bounds
for $n=7$ and $r=4$.

Given a graph $G$ of order $n$, we are interested in the following inequalities, one of them being a conjecture, the other ones being proved here below:
$$
L_1(n) {\le} {\min}\Big\{L_2(n,\chi(G)), L_3(n,\Delta(G){+}1)\Big\} {\le} {\max}\Big \{L_2(n,\chi(G)), L_3(n,\Delta(G){+}1)\Big \} {\le} \A(G).
$$

The first inequality follows from Theorem~\ref{thm:removeSimplicialEdge} since $\E{n}$ is obtained from $\K{r} {\cup} (n{-}r) \K{1}$ and from $\K{1,r{-}1} {\cup} (n{-}r) \K{1}$ by repeatedly removing edges incident to simplicial vertices. The second inequality is trivial.
The last inequality is an open problem stated in the two following conjectures.

\begin{conj} \label{conj_chi}
Let  $G$ be a graph of order $n$. Then,
$$
\avcol(G)\geq L_2(n,\chi(G))
$$
with equality if and only if $G \simeq \K{\chi(G)} \cup (n{-}\chi(G)) \K{1}$.
\end{conj}

\begin{conj} \label{conj_delta}
Let  $G$ be a graph of order $n$. Then
$$
\avcol(G) \ge L_3(n,\Delta(G){+}1)
$$
with equality if and only if $G \simeq \K{1,\Delta(G)} \cup (n{-}\Delta(G){-}1)\K{1}$.
\end{conj}

Since $L_1(n) \le \min\{L_2(n,\chi(G)), L_3(n,\Delta(G){+}1)\}$, it suffices to show that one of these conjectures is true to prove that the empty graph $\E{n}$ has the minimum value for $\A(G)$ among all graphs $G$ of order $n$. This leads to the following weaker conjecture.

\begin{conj} \label{conj_empty}
Let  $G$ be a graph of order $n$, then,
$$
\avcol(G)\ge L_1(n)
$$
with equality if and only if $G \simeq \E{n}$.
\end{conj}

\begin{thm}
Conjectures \ref{conj_chi} and \ref{conj_delta} (and therefore \ref{conj_empty}) are true for triangulated graphs.
\end{thm}
\begin{proof} Let us first observe that removing an edge incident to a simplicial vertex in a triangulated graph gives another triangulated graph.
So let $G$ be a triangulated graph. Since $G$ is perfect, it contains a clique $K$ of order $|K|=\chi(G)$. It is well known that  triangulated graphs that are not a clique contain at least two non-adjacent simplicial vertices \cite{Dirac}. Hence, $G$ can be reduced to $\K{\chi(G)} \cup (n{-}\chi(G)) \K{1}$ by repeatedly removing  edges incident to simplicial vertices. We know from Theorem  \ref{thm:removeSimplicialEdge} that each of these edge removals strictly decreases $\avcol(G)$. We thus have  $\avcol(G)\geq \avcol(\K{\chi(G)} \cup (n{-}\chi(G)) \K{1})$, with equality if and only if $G\simeq \K{\chi(G)} \cup (n{-}\chi(G)) \K{1}$. Conjecture \ref{conj_chi} is therefore true for triangulated graphs.
	
	Let us now deal with Conjecture \ref{conj_delta}. Let $v$ be a vertex of degree $\Delta(G)$ in $G$. We consider the partition $(N_1(v),N_2(v))$ of the neighborhood $N(v)$ of $v$, where $N_1(v)$ contains all vertices of $N(v)$ of degree 1 in $G$ (i.e., $v$ is the only neighbor of every vertex of $N_1(v)$). Also, we consider the partition $(\overline{N}_1(v),\overline{N}_2(v))$ of the set $\overline{N}(v)$ of vertices of $G$ that are not adjacent to $v$, where $\overline{N}_1(v)$ contains all vertices of $\overline{N}(v)$ of degree 0 in $G$. If $N_2(v)\cup \overline{N}_2(v)\neq \emptyset$ then $G[N_2(v)\cup \overline{N}_2(v)\cup\{v\}]$ contains a simplicial vertex $w\neq v$ (since it is also a triangulated graph which is not a clique). Clearly, $w$ is simplicial in the whole graph that includes $N_1(v)$ and $\overline{N}_1(v)$. Then :
\begin{itemize}\setlength\itemsep{-3pt}
		\item If $w\in N_2(G)$, we can remove all edges incident to $w$, except the one that links $w$ with $v$. We thus get a new triangulated graph in which 
		at least one vertex has been transferred from $N_2(v)$ to $N_1(v)$,  vertices of $\overline{N}_2(v)$ may have transferred to $\overline{N}_1(v)$, but no vertex has undergone the reverse transfers.
		\item If $w\in \overline{N}_2(v)$, we can remove all edges incident to $w$. We thus get a new triangulated graph in which at least one vertex has been transferred from $\overline{N}_2(v)$ to $\overline{N}_1(v)$,  vertices of $N_2(v)$ may have transferred to $N_1(v)$, but no vertex has undergone the reverse transfers.
	\end{itemize}
	Note that in both cases, no vertex has been transferred from $\overline{N}(v)$ to $N(v)$ or vice versa. Hence, by repeatedly applying the above mentioned  edge removals, we get $N_2(v)=\overline{N}_2(v)=\emptyset$, which means that the resulting graph is $ \K{1,\Delta(G)} \cup (n{-}\Delta(G)-1)\K{1}$. Again, we know from Theorem  \ref{thm:removeSimplicialEdge} that each of the edge removals performed strictly decreases $\avcol(G)$, which proves  that $\avcol(G)\geq \avcol(\K{1,\Delta(G)} \cup (n{-}\Delta(G)-1)\K{1})$, with equality if and only if $G\simeq \K{1,\Delta(G)} \cup (n{-}\Delta(G)-1)\K{1}$.
\end{proof}

The three conjectures come from the discovery systems \GPH~\cite{GPH} and \linebreak\PHOEG~\cite{PHOEG}.
Note that despite the apparent simplicity of Conjecture~\ref{conj_empty}, its validity cannot be proven by simple intuitive means such as sequential edge removal.
Indeed, there are graphs, for example $\K{2,4}$, for which the removal of any edge strictly increases $\avcol(G)$. Also, we cannot proceed by induction on the the number of connected components of $G$. Indeed, there are pairs of graphs $G_1,G_2$ such that $\avcol(G_1)<\avcol(G_2)$ while $\avcol(G_1{\cup}\K{1})>\avcol(G_2{\cup}\K{1})$. For example, for $G_1=\K{2,3}$ and $G_2=\K{3}{\cup}2\K{1}$, we have
$$
\avcol(G_1)=3.5<3.529=\avcol(G_2)\quad\quad\mbox{and}\quad\quad \avcol(G_1{\cup}\K{1})=3.867>3.831=\avcol(G_2{\cup}\K{1}).
$$

Note that proving that Conjecture~\ref{conj_delta} is true for all graphs $G$ of order $n$ and maximum degree $\Delta(G)=n-1$ is as difficult as proving Conjecture~\ref{conj_empty}. Indeed, let $v$ be a vertex of degree $n-1$ in a graph $G$ of order $n$. Since $v$ is a dominant vertex of $G$, we know from Corollary~\ref{pro_dominant} that $\avcol(G)=\avcol(G-v)$+1. Hence, minimizing $\avcol(G)$ is equivalent to minimizing $\avcol(G-v)$, with no maximum degree constraint on $G-v$. We show in the next section that Conjectures~\ref{conj_chi} and~\ref{conj_delta} (and therefore \ref{conj_empty}) are true for graphs of maximum degree at most 2.

\subsection{Proof of the conjectures for graphs $G$ with $\Delta(G)\le 2$}

We start this section with a simple proof of the validity of Conjectures~\ref{conj_chi} and~\ref{conj_delta} when $\Delta(G)=1$.

\begin{thm} \label{thm:lb_delta1}
	Let $G$ be a graph of order $n$ and maximum degree $\Delta(G) = 1$. Then,
	$$
	L_2(n,\chi(G)) = L_3(n,\Delta(G){+}1) \le \A(G),
	$$
	with equality if and only if $G \simeq \K{2} {\cup} (n{-}2) \K{1}$.
\end{thm}

\begin{proof}
	Since $\Delta(G) = 1$, we have $\chi(G) = 2$, which implies $$L_2(n,\chi(G)){=}L_2(n,2){=}\avcol(\K{2} {\cup} (n{-}2) \K{1}){=}\avcol(\K{1,1} {\cup} (n{-}2) \K{1})=L_3(n,2){=}L_3(n,\Delta(G){+}1).$$
	Note also that $\Delta(G) = 1$ implies $G \simeq p \K{2} {\cup} (n{-}2p) \K{1}$ for $p \ge 1$. Hence, all vertices in $G$ are simplicial. We can thus sequentially remove all edges of $G$, except one, and it follows from Theorem ~\ref{thm:removeSimplicialEdge} that $\avcol(G)\ge \avcol(\K{2} {\cup} (n{-}2) \K{1})$, with equality if and only $G \simeq \K{2} {\cup} (n{-}2)\K{1}$.
\end{proof}

The proofs that Conjectures~\ref{conj_chi} and~\ref{conj_delta} are true when $\Delta(G)=2$ are more complex. We first prove some intermediate results in the form of lemmas.

\begin{lem}\label{lem:12}
	$\A(G\cup \C{n})>\A(G\cup \Path{n})$ for all $n\geq 3$ and all graphs $G$.
\end{lem}
\begin{proof}
	Let $H\simeq \Path{2}$ if $n=3$ and $H\simeq\C{n-1}$ if $n>3$. We know from Equations (\ref{rec_minus}) that $\ncol(G\cup \C{n})=\ncol(G\cup \Path{n})-\ncol(G\cup H)$ and $\totcol(G\cup \C{n})=\totcol(G\cup \Path{n})-\totcol(G\cup H)$.  Since $\Path{n}$ is a partial subgraph of $\C{n}$, we have $\ncol(G\cup \C{n})>\ncol(G\cup \Path{n})$. Altogether, this gives
	\begin{align*}
	&\avcol(G\cup \C{n})-\A(G\cup H)
	&&=\frac{\totcol(G\cup\C{n})}{\ncol(G\cup\C{n})}-\frac{\totcol(G\cup H)}{\ncol(G\cup H)}\\
	&&&=
	\frac{\totcol(G\cup \Path{n})-\totcol(G\cup H)}{\ncol(G\cup \Path{n})-\ncol(G\cup H)}-\frac{\totcol(G\cup H)}{\ncol(G\cup H)}\\
	&&&=\frac{\totcol(G\cup \Path{n})\ncol(G\cup H)-\totcol(G\cup H)\ncol(G\cup \Path{n})}{\ncol(G\cup \C{n})\ncol(G\cup H)}\\
	&&&>\frac{\totcol(G\cup \Path{n})\ncol(G\cup H)-\totcol(G\cup H)\ncol(G\cup \Path{n})}{\ncol(G\cup \Path{n})\ncol(G\cup H)}\\
	&&&=\frac{\totcol(G\cup\Path{n})}{\ncol(G\cup\Path{n})}-\frac{\totcol(G\cup H)}{\ncol(G\cup H)}\\
	&&&=
	\avcol(G\cup \Path{n})-\A(G\cup H)\\
	\iff&\A(G\cup \C{n})>\A(G\cup \Path{n}).&&\qedhere
	\end{align*}
\end{proof}

For $n\geq 3$, let $\Q{n}$ be the graph obtained from $\Path{n}$ by adding an edge between an extremity $v$ of $\Path{n}$ and the vertex at distance 2 from $v$ on $\Path{n}$.

\begin{lem} \label{prop:ncolkPath+}
	If $n\geq 3$, $0\leq x\leq p$ and $1\le k\le n$, then
	$$\ncolk(\Q{n} \cup p\K{1},k) = \sum_{i=0}^x {x \choose i}\ncolk(\Q{n+i} \cup (p-x)\K{1},k).$$
\end{lem}

\begin{proof}
	The result is clearly true for $p=0$. For larger values of $p$, we proceed by induction. Since the result is clearly true for $x=0$, we assume $x\geq 1$. Equations \eqref{recs_plus} imply
	 \allowdisplaybreaks
	\begin{align*}
	\ncolk(\Q{n} \cup p\K{1},k) &= \ncolk(\Q{n+1} \cup (p{-}1)\K{1},k) + \ncolk(\Q{n} \cup (p{-}1)\K{1},k) \\
	=& \sum_{i=0}^{x-1} {x{-}1 \choose i}\ncolk(\Q{n+i+1} \cup (p{-}x)\K{1},k) + \sum_{i=0}^{x-1} {x{-}1 \choose i}\ncolk(\Q{n+i} \cup (p{-}x)\K{1},k) \\
	=& \sum_{i=1}^{x} {x{-}1 \choose i{-}1}\ncolk(\Q{n+i} \cup (p{-}x)\K{1},k) + \sum_{i=0}^{x-1} {x{-}1 \choose i}\ncolk(\Q{n+i} \cup (p{-}x)\K{1},k)\\
	=& \ncolk(\Q{n{+}x} {\cup} (p{-}x)\K{1},k) {+} \ncolk(\Q{n} {\cup} (p{-}x)\K{1},k) \\
	&+ \sum_{i=1}^{x-1} \left({x{-}1 \choose i{-}1}+{x{-}1 \choose i}\right)\ncolk(\Q{n{+}i} {\cup} (p{-}x)\K{1},k)\\
	=& \sum_{i=0}^{x} {x \choose i} \ncolk(\Q{n+i} \cup (p{-}x)\K{1},k).\vspace{-1cm}\qedhere
	\end{align*}
\end{proof}

\begin{lem} \label{prop:ncolkCn3}
	If $n\geq 3$ is an odd number and $1\leq k\le n$, then 
	$$\ncolk(\C{n} \cup p\K{1},k) = \sum_{i=0}^{(n-3)/2}\ncolk(\Q{2i+3} \cup p\K{1},k).$$
\end{lem}

\begin{proof}
	The result is clearly true for $n=3$ since $\C{3}\simeq \Q{3}$. For larger values of $n$, we proceed by induction. It follows from Equations \eqref{recs_minus} and \eqref{recs_plus} that	
	\begin{align*}
		\ncolk(\C{n} \cup p\K{1},k) &=\ncolk(\Path{n} \cup p\K{1})-\ncolk(\C{n-1} \cup p\K{1},k)\\
	&=\Bigl(\ncolk(\Q{n} {\cup} p\K{1},k){+}\ncolk(\Path{n{-}1}{\cup} p\K{1},k)\Bigr){-}\Bigl(\ncolk(\Path{n{-}1} {\cup} p\K{1},k){-}\ncolk(\C{n{-}2} {\cup} p\K{1},k)\Bigr)\\
	&=\ncolk(\Q{n} \cup p\K{1},k)+\ncolk(\C{n-2} \cup p\K{1},k)\\
	&= \ncolk(\Q{n} \cup p\K{1},k) + \sum_{i=0}^{(n-5)/2}\ncolk(\Q{2i+3} \cup p\K{1},k) \\
	&= \sum_{i=0}^{(n-3)/2}\ncolk(\Q{2i+3} \cup p\K{1},k).
	\end{align*}
\end{proof}

\begin{lem} \label{cor:ncolkC3K1}
	If $n$ and $x$ are two numbers such that $5\leq x\leq n$ and $x$ is odd, then
	$$\smash\ncolk(\C{3} \cup (n-3)\K{1},k) = \ncolk(\C{x} \cup (n-x)\K{1},k) + \sum_{i=0}^{x-5} \alpha_i \ncolk(\Q{i+4} \cup (n-x)\K{1},k)$$
	where 
	$$\alpha_i=\begin{cases}
	\displaystyle{x-3 \choose i}-1&\mbox{if }i\mbox{ is even}\\[0.5cm]
	\displaystyle {x-3 \choose i}&\mbox{if }i\mbox{ is odd.}
	\end{cases}$$
	
\end{lem}

\begin{proof} Since $\Q{3}\simeq\C{3}$ we know from Lemma \ref{prop:ncolkPath+} that
\begin{align*}
\ncolk(\Q{3} \cup (n-3)\K{1},k) =& \sum_{i=0}^{x-3} {x-3 \choose i}\ncolk(\Q{i+3} {\cup} (n-x)\K{1},k)\\
=& \sum_{i=0}^{(x{-}3)/2}{x{-}3 \choose 2i}\ncolk(\Q{2i+3} {\cup} (n{-}x)\K{1},k) \\
&{+} \sum_{i=0}^{(x{-}5)/2}{x{-}3 \choose 2i+1}\ncolk(\Q{2i+4} \cup (n{-}x)\K{1},k).
\end{align*}

\noindent It then follows from Lemma \ref{prop:ncolkCn3} that

\begin{align*}
\ncolk(\Q{3} {\cup} (n{-}3)\K{1},k) {=} & S(\C{x}{\cup} (n{-}x)\K{1},k)
	{+} \sum_{i=1}^{(x{-}5)/2}\left({x{-}3 \choose 2i}{-}1\right)\ncolk(\Q{2i+3} {\cup} (n{-}x)\K{1},k)\\
	& + \sum_{i=0}^{(x-5)/2}{x-3 \choose 2i+1}\ncolk(\Q{2i+4} \cup (n-x)\K{1},k) \\
	= & \ncolk(\C{x} \cup (n-x)\K{1},k) + \sum_{i=0}^{x-5} \alpha_i \ncolk(\Q{i+4} \cup (n-x)\K{1},k).
	\end{align*}
	
\end{proof}

\begin{lem} \label{cor:avcolPi+Cn}
	If $n\geq 5$ and $3\leq i<n$ then
	$$\avcol(\Q{i} \cup p\K{1}) < \avcol(\C{n} \cup p\K{1}) $$
\end{lem}
\begin{proof}
	Since $\Q{n-1}\cup p\K{1}$ is obtained from $\Q{i} \cup p\K{1}$ by iteratively adding vertices of degree $1$, we know from Corollary \ref{cor:5} that $\avcol(\Q{i}\cup p\K{1}) \leq \avcol(\Q{n-1} \cup p\K{1})$. Moreover, it is proved in \cite{Hertz21} that $\avcol(\Q{n-1} \cup p\K{1}) < \avcol(\Path{n} \cup p\K{1})$ for all $n\geq 5$ and $p\geq 0$. It then follows from Lemma \ref{lem:12} that
	$$\avcol(\Q{i}\cup p\K{1}) \leq \avcol(\Q{n-1} \cup p\K{1}) < \avcol(\Path{n} \cup p\K{1})< \avcol(\C{n} \cup p\K{1}).$$
\end{proof}

\begin{cor} \label{cor:avcolC3Cx}
	If  $n\geq 5$, $x$ is odd and $5\leq x\leq n$, then 
	$$\avcol(\C{3} \cup (n-3)\K{1}) < \avcol(\C{x} \cup (n-x)\K{1}).$$
\end{cor}

\begin{proof} Lemma \ref{cor:ncolkC3K1} implies
\begin{itemize}
\item $\ncol(\C{3} \cup (n-3)\K{1}) = \ncol(\C{x} \cup (n-x)\K{1}) + \displaystyle\sum_{i=0}^{x-5} \alpha_i \ncol(\Q{i+4} \cup (n-x)\K{1})$, and
\item $\totcol(\C{3} \cup (n-3)\K{1}) = \totcol(\C{x} \cup (n-x)\K{1}) + \displaystyle\sum_{i=0}^{x-5} \alpha_i \totcol(\Q{i+4} \cup (n-x)\K{1})$,
\end{itemize}
	where \begin{itemize}
	    \item $\alpha_i=\displaystyle{x-3 \choose i}-1\geq 0$ if $i$ is even, and
	   \item $\alpha_i=\displaystyle{x-3 \choose i}> 0$ if $i$ is odd.
	\end{itemize} 
		Also, we know from Lemma \ref{cor:avcolPi+Cn} that 
$\avcol(\Q{i+4} \cup p\K{1}) {<} \avcol(\C{x} \cup (n-x)\K{1}) $ for $i=0,\ldots,x{-}5$. Hence, it follows from Theorem~\ref{thm:pro_avcolInfSum} that $\avcol(\C{3} \cup (n-3)\K{1}) < \avcol(\C{x} \cup (n-x)\K{1})$.

\end{proof}

\noindent We are now ready to prove the validity of  Conjectures \ref{conj_chi} and \ref{conj_delta} when $\Delta(G) =2$.

\begin{thm} \label{thm:lb_delta2a}
	Let $G$ be a graph of order $n$ with $\Delta(G)=2$. 
	Then,
	$$\avcol(G)\geq L_2(n,\chi(G))
	$$
	with equality if and only if $G \simeq \K{\chi(G)} \cup (n{-}\chi(G)) \K{1}$.
\end{thm}

\begin{proof}
	Since $\Delta(G)=2$, $G$ is the disjoint union of paths and cycles.  If $G$ does not contain any odd cycle, then $\chi(G)=2$. It then follows from Theorem~\ref{thm:removeSimplicialEdge} and Lemma~\ref{lem:12} that the edges of $G$ can be removed sequentially, with a strict decrease of $\avcol(G)$ at each step, until we get $\K{2} \cup (n{-}2) \K{1}$.
	
	If $\chi(G) {=} 3$, then at least one connected component of $G$ is an odd cycle $\C{x}$ with $x{\leq} n$.  Again, we know from Theorem~\ref{thm:removeSimplicialEdge} and Lemma~\ref{lem:12} that the edges of $G$ can be removed sequentially, with a strict decrease of $\avcol(G)$ at each step, until we get $\C{x} \cup (n{-}x) \K{1}$. It then follows from Corollary~\ref{cor:avcolC3Cx} that $\avcol(G)\geq\avcol(\C{x} \cup (n{-}x) \K{1})\geq \avcol(\C{3} \cup (n{-}3) \K{1})$, with equalities if and only if $G\simeq \C{3} \cup (n{-}3) \K{1}\simeq \K{3} \cup (n{-}3) \K{1}$.
	
\end{proof}

\begin{thm} \label{thm:lb_delta2b}
	Let $G$ be a graph of order $n$ with $\Delta(G)=2$. 
	Then,
	$$\avcol(G)\geq L_3(n,\Delta(G)+1)
	$$
	with equality if and only if $G \simeq \K{1,2} \cup (n{-}3) \K{1}$.
\end{thm}

\begin{proof}
	Since $\Delta(G)=2$, $G$ is the disjoint union of paths and cycles. Also, $G$ contains at least one vertex $u$ of degree 2. Let $v$ and $w$ be two neighbors of $u$ in $G$. It follows from Theorem~\ref{thm:removeSimplicialEdge} and Lemma~\ref{lem:12} that the edges of $G$ can be removed sequentially, with a strict decrease of $\avcol(G)$ at each step, until the edge set of the remaining graph $H$ is $\{uv,uw\}$. But $H$ is then isomorphic to $\K{1,2} \cup (n{-}3) \K{1}$.
	\end{proof}

\section{Concluding remarks}

We have established several properties for a recently defined graph invariant, namely the average number $\avcol(G)$ of colors in the non-equivalent colorings of a graph $G$. We then looked at bounds for $\avcol(G)$. It is easy to prove that $\avcol(G)\leq \avcol(\K{n})=n$ for all graphs of order $n$, with equality if and only if $G\simeq \K{n}$. Hence, $n$ is the best possible upper bound on $\avcol(G)$ for a graph $G$ of order $n$. We think that the best possible lower bound on $\avcol(G)$ for a graph $G$ of order $n$ is $\avcol(\E{n})$. We have shown that despite its apparent simplicity, this conjecture cannot be proven using simple techniques like sequential edge removal. We have then refined this conjecture by proposing lower bounds related to the chromatic number $\chi(G)$ and to the maximum degree $\Delta(G)$ of $G$. We have thus stated three open problems. We have  shown that these three conjectures are true for triangulated graphs and for graphs with maximum degree at most 2.




\bibliographystyle{acm}
\bibliography{avcolLB}

\end{document}